\documentclass[12pt]{article}
\usepackage{amsfonts}
\usepackage[pctex32]{graphics}
\usepackage[dvips]{graphicx}
\usepackage{amsthm}
\usepackage{amsmath}
\usepackage{amssymb}
\usepackage{fontenc}
\usepackage[all]{xy}
\usepackage{latexsym}
\usepackage{layout}
\usepackage{epsfig}
\usepackage{graphicx}
\usepackage{pstricks,pst-node,pst-plot} 
\usepackage{color} 
\title{Distribution of the Maximum and Minimum of a Random Number of Bounded Random Variables}
\author {Jie Hao and Anant Godbole\\
Department of Mathematics and Statistics\\
East Tennessee State University}
\begin{document}
\def\qed{\vbox{\hrule\hbox{\vrule\kern3pt\vbox{\kern6pt}\kern3pt\vrule}\hrule}}
\def\ms{\medskip}
\def\n{\noindent}
\def\ep{\varepsilon}
\def\G{\Gamma}
\def\lr{\left(}
\def\ls{\left[}
\def\rs{\right]}
\def\lf{\lfloor}
\def\rf{\rfloor}
\def\lg{{\rm lg}}
\def\lc{\left\{}
\def\rc{\right\}}
\def\rr{\right)}
\def\ph{\varphi}
\def\p{\mathbb P}
\def\nk{n \choose k}
\def\a{\cal A}
\def\s{\cal S}
\def\e{\mathbb E}
\def\v{\mathbb V}
\def\l{\lambda}
\newcommand{\bsno}{\bigskip\noindent}
\newcommand{\msno}{\medskip\noindent}
\newcommand{\oM}{M}
\newcommand{\omni}{\omega(k,a)}
\newtheorem{thm}{Theorem}[section]
\newtheorem{con}{Conjecture}[section]
\newtheorem{claim}[thm]{Claim}
\newtheorem{deff}[thm]{Definition}
\newtheorem{lem}[thm]{Lemma}
\newtheorem{cor}[thm]{Corollary}
\newtheorem{rem}[thm]{Remark}
\newtheorem{prop}[thm]{Proposition}
\newtheorem{ex}[thm]{Example}
\newtheorem{eq}[thm]{equation}
\newtheorem{que}{Problem}[section]
\newtheorem{ques}[thm]{Question}
\providecommand{\floor}[1]{\left\lfloor#1\right\rfloor}
\maketitle
\begin{abstract}
We study a new family of random variables, that each arise as the distribution of the maximum or minimum of a random number $N$ of i.i.d.~random variables $X_1,X_2,\ldots,X_N$, each distributed as a variable $X$ with support on $[0,1]$.  The general scheme is first outlined, and several special cases are studied in detail.  Wherever appropriate, we find estimates of the parameter $\theta$ in the one-parameter family in question.
\end{abstract}

\section{Introduction and General Scheme}    Consider a sequence $X_1, X_2,\ldots$ of i.i.d.~random variables with support on $[0,1]$ and having distribution function $F$.  For any fixed $n$, the distributions of
\[Y=\max_{1\le i\le n}X_i\]
and
\[Z=\min_{1\le i\le n}X_i\]
have been well studied; in fact it is shown in elementary texts that $F_Y(x)=F^n(x)$ and $F_Z(x)=1-(1-F(x))^n$.  But what if we have a situation where the number $N$ of $X_i$'s is random, and we are instead considering the extrema
\begin{equation}Y=\max_{1\le i\le N}X_i\end{equation}
and
\begin{equation}Z=\min_{1\le i\le N}X_i\end{equation}
of a random number of i.i.d.~random variables?  Now the {\it sum} $S$ of a random number of i.i.d.~variables, defined as
\[S=\sum_{i=1}^N X_i\]
satisfies, according to Wald's Lemma \cite{durrett}, the equation
\[\e(S)=\e(N)\e(X),\]
provided that $N$ is independent of the sequence $\{X_i\}$ and assuming that the means of $X$ and $N$ exist.  The purpose of this paper is to show that the distributions in (1) and (2) can be studied in many canonical cases, even if $N$ and $\{X_i\}_{i=1}^\infty$ are correlated.  The main deviation from the papers \cite{expo}, \cite{expo2} and \cite{weib}, where similar questions are studied, is that the variable $X$ is concentrated on the interval $[0,1]$ -- unlike the above references, where $X$ has lifetime-like distributions on $[0,\infty)$.  Even then, we find that many new and interesting distributions arise, none of them to be found, e.g., in \cite{bala} or \cite{bb}.  In another deviation from the theory of extremes of random sequences (see, e.g., \cite{extr}), we find that the tail behavior of the extreme distributions is not relevant due to the fact that the distributions have compact support.  We next cite three examples where our methods might be useful.  First, we might be interested in the strongest earthquake in a given region in a given year.  The number of earthquakes in a year, $N$, is usually modeled using a Poisson distribution, and, ignoring aftershocks and similarly correlated events, the intensities of the earthquakes can be considered to be i.i.d.~random variables in $[a,b]$ whose distribution can be modeled using, e.g., the data set maintained by Caltech at \cite{earth}.  Second, many ``small world" phenomena have recently been modeled by power law distributions, also sometimes termed discrete Pareto or Zipf distributions.  See, for example, the body of work by Chung and her co-authors \cite {chung1}, \cite{chung2}, and the references therein, where vertex degrees $d(v)$ in ``internet-like graphs" $G$ (e.g., the vertices of $G$ are individual webpages, and there is an edge between $v_1$ and $v_2$ if one of the webpages has a link to the other) are shown to be modeled by 
\[\p(d(v)=n)=\frac{[\zeta(k)]^{-1}}{n^k}\]
for some constant $k>1$, where $\zeta(\cdot)$ is the Riemann Zeta function
\[\zeta(k)=\sum_{n=1}^\infty\frac{1}{n^k}.\]  Thus if the vertices $v$ in a large internet graph have some bounded i.i.d.~property $X_i$, then the maximum and minimum values of $X_i$ for the neighbors of a randomly chosen vertex can be modeled using the methods of this paper.  Third, we note that $N$ and the $X_i$ may be correlated, as in the CSUG example (studied systematically in Section 3) where $X_i\sim U[0,1]$ and $N=\inf\{n\ge2: X_n>(1-\theta)\}$ follows the geometric distribution ${\rm Geo}(\theta)$.  This is an example of a situation where we might be modeling the maximum load that a device might have carried before it breaks down due to an excessive weight or current.  It is also feasible in this case that the parameter $\theta$ might be unknown.

Here is our general set-up:  Suppose $X_1, X_2, \dots$ are i.i.d.~random variables following a continuous distribution on $[0,1]$ with probability density and distribution functions given by $f(x)$ and $F(x)$ respectively. $N$ is a random variable following a discrete distribution on $\{1,2,\ldots\}$ with probability mass function given by $\p(N=n)=p(n),~n=1,2,\dotsc$. Let $Y$ and $Z$ be given by (1) and (2) respectively. Then the p.d.f.'s $g$ of $Y$ and $Z$ are derived as follows:
Since\[\p(Y\le y| N=n)=[F(y)]^n,\]
we see that\[g(y|N=n)=n[F(y)]^{n-1}f(y),\]
and consequently, the marginal p.d.f.~of $Y$ is 
\begin{eqnarray}
g(y)&=&\sum_{n=1}^\infty g(y|N=n)\p(N=n) \nonumber\\
&=& f(y)\sum_{n=1}^\infty n[F(y)]^{n-1}p(n). \end{eqnarray}
In a similar fashion, the p.d.f. of $Z$ can be shown to be \begin{equation}g(z)=f(z)\sum_{n=1}^\infty n[1-F(z)]^{n-1}p(n);\end{equation}
what is remarkable is that the sums in (3) and (4) will be shown to assume simple tractable forms in a variety of cases.

Our paper is organized as follows.  In Section 2, we study the case of $X\sim U[0,1]$ and $N\sim{\rm Geo}(\theta)$.  We call this the Standard Uniform Geometric model.  The CSUG (Correlated Standard Uniform Model) is studied in Section 3.  Section 4 is devoted to a summary of a variety of other models.
\section{Standard Uniform Geometric (SUG) Model}    
Since $f(x)=1, 0\le x \le 1$, and 
$\p(N=n)=\theta(1-\theta)^{n-1}\enspace (n=1,2,\ldots)$  for some $\theta\in(0,1)$, we have from (3) that the p.d.f.~of $Y$ in the SUG Model is given by
\begin{eqnarray}g(y)&=&\sum_{n=1}^{\infty}\theta(1-\theta)^{n-1}\times ny^{n-1}\nonumber\\
&=&\frac{\theta}{[1-(1-\theta)y]^2}.
\end{eqnarray}    
\begin{figure}[tbp] 
  \centering
  \includegraphics[bb=132 223 480 569,width=4.68in,height=3.30in,keepaspectratio]{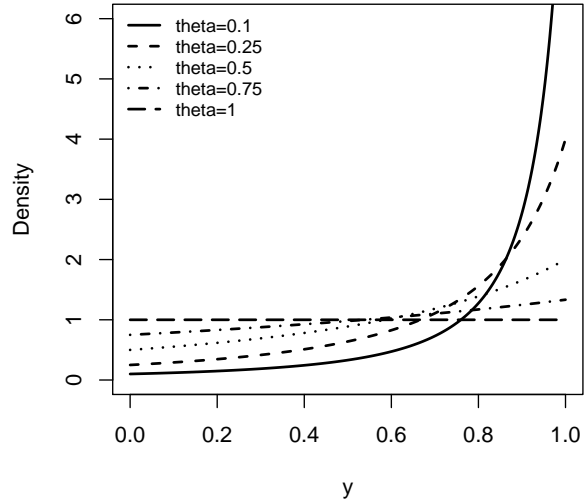}
  \caption{Plot of the SUG maximum density for some values of $\theta$}
  \label{fig:sugy}
\end{figure}
Similarly, (4) gives that
\begin{eqnarray}g(z)&=&\sum_{n=1}^{\infty}\theta(1-\theta)^{n-1}\times n(1-z)^{n-1}\nonumber\\
&=&\frac{\theta}{[1-(1-\theta)(1-z)]^2}\nonumber\\
&=&\frac{\theta}{[\theta+(1-\theta)z]^2}.
\end{eqnarray}
\begin{figure}[tbp] 
  \centering
  \includegraphics[bb=132 223 480 569,width=4.68in,height=3.30in,keepaspectratio]{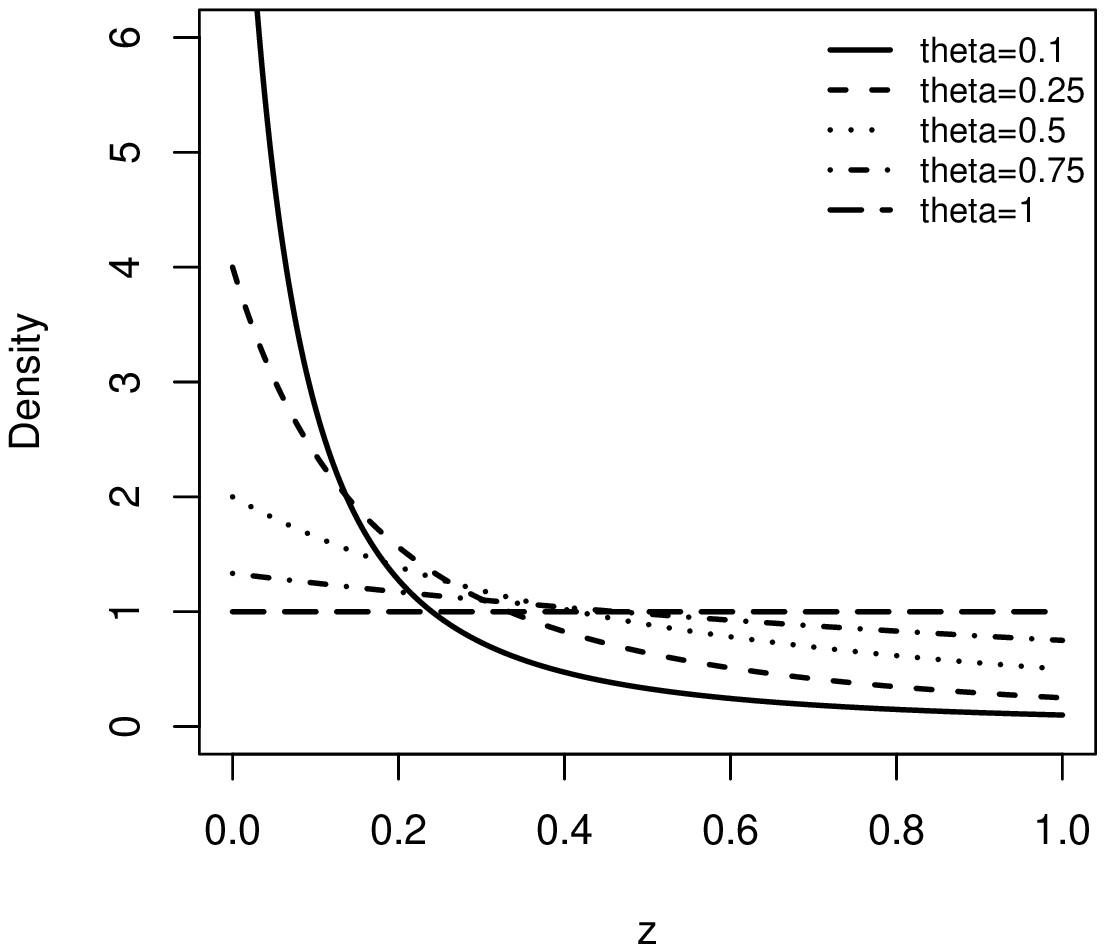}
  \caption{Plot of the SUG minimum density for some values of $\theta$}
  \label{fig:sugz}
\end{figure}
\begin{prop} If the random variable $Y$ has the ``SUG maximum distribution" (5) and $k\in {\mathbb N}$, then
$$\e(Y^k)=\frac{\theta}{(1-\theta)^{k+1}}\sum_{j=0}^{k}{k \choose j}\int_\theta^1(-u)^{j-2}du.$$ \end{prop}
\begin{proof}
\begin{eqnarray*}\e(Y^k)&=&\int_0^1 y^k \times \frac{\theta}{[1-(1-\theta)y]^2}dy\\
&=&\int_1^\theta \lr\frac{1-u}{1-\theta}\rr^k \times \frac{\theta}{u^2}\times \lr-\frac{1}{1-\theta}\rr du\\
&=&\frac{\theta}{(1-\theta)^{k+1}}\int_\theta^1 \frac{(1-u)^k}{u^2}du\\
&=&\frac{\theta}{(1-\theta)^{k+1}}\int_\theta^1 \frac{\sum_{j=0}^k {k \choose j}(-u)^j}{u^2}du\\
&=&\frac{\theta}{(1-\theta)^{k+1}}\sum_{j=0}^{k}{k \choose j}\int_\theta^1(-u)^{j-2}du, \end{eqnarray*}
as claimed.    \hfill\end{proof}
\begin{prop} The random variable $Y$ has mean and variance given, respectively, by
$$\e(Y)=\frac{\theta(\ln \theta+\displaystyle\frac{1}{\theta}-1)}{(1-\theta)^2}~~~\mbox{and}~~~\v(Y)=\frac{\theta^3-2\theta^2-\theta^2\ln^2\theta +\theta}{(1-\theta)^4}.$$ \end{prop}
\begin{proof}Using Proposition 2.1, we can directly compute the mean and variance by setting $k=1,2$, and using the fact that $\v(W)=\e(W^2)-[\e(W)]^2$ for any random variable $W$. (This proof could equally well have been based on calculating the moments of $1-(1-\theta)Y$ and then recovering the values of $\e(Y)$ and $\v(Y)$.  The same is true of other proofs in the paper.)\end{proof}
\begin{prop} If the random variable $Z$ has the ``SUG minimum distribution" and $k\in {\mathbb N}$, then
$$\e(Z^k)=\frac{\theta}{(1-\theta)^{k+1}}\sum_{j=0}^{k}{k \choose j}(-\theta)^j\int_\theta^1 u^{k-j-2}du.$$ \end{prop}
\begin{proof}
\begin{eqnarray*}\e(Z^k)&=&\int_0^1 z^k \times \frac{\theta}{[\theta+(1-\theta)z]^2}dz\\
&=&\int_\theta^1 \lr\frac{u-\theta}{1-\theta}\rr^k \times \frac{\theta}{u^2}\times \frac{1}{1-\theta}du\\
&=&\frac{\theta}{(1-\theta)^{k+1}}\int_\theta^1 \frac{(u-\theta)^k}{u^2}du\\
&=&\frac{\theta}{(1-\theta)^{k+1}}\int_\theta^1 \frac{\sum_{j=0}^k {k \choose j}u^{k-j}(-\theta)^j}{u^2}du\\
&=&\frac{\theta}{(1-\theta)^{k+1}}\sum_{j=0}^{k}{k \choose j}(-\theta)^j\int_\theta^1 u^{k-j-2}du, \end{eqnarray*}
as asserted.\hfill\end{proof}
\begin{prop} The random variable $Z$ has mean and variance given, respectively, by
$$\e(Z)=\frac{\theta(\theta-1-\ln \theta)}{(1-\theta)^2}~~~\mbox{and}~~~\v(Z)=\frac{\theta^3-2\theta^2-\theta^2\ln^2\theta +\theta}{(1-\theta)^4}.$$ \end{prop}

\begin{proof}Using Proposition 2.3, it is easily to compute the mean and variance by setting $k=1,k=2$.  \hfill\end{proof}

The m.g.f.'s of $Y,Z$ are easy to calculate too.  Notice that the logarithmic terms above arise due to the contributions of the $j=1$ and $j=k-1$ terms, and it is precisely these logarithmic terms that make, e.g., method of moments estimates for $\theta$ to be intractable in a closed (i.e., non-numerical) form.  Similar difficulties arise when analyzing the likelihood function and likelihood ratios.
\section{The Correlated Standard Uniform Geometric (CSUG) Model}

The Correlated Standard Uniform Geometric (CSUG) model is related to the SUG model, as the name suggests, but $X$ and $N$ are correlated as indicated in Section 1. The CSUG problems arise in two cases. One case is that we conduct standard uniform trials until a variable $X_i$ exceeds $1-\theta$, where $\theta$ is the parameter of the correlated geometric variable, and the maximum of $X_1, X_2,\cdots, X_{i-1}$ is what we seek. The maximum is between 0 and $1-\theta$. The other case is where standard uniform trials are conducted until $X_i$ is less than $\theta$, and we are looking for the minimum of $X_1, X_2,\cdots, X_{i-1}$. The minimum is between $\theta$ and 1.

Specifically, let $X_1, X_2,\cdots$ be a sequence of standard uniform variables
and define $$N=\inf\{n\ge 2:X_i>1-\theta\},$$ or $$N=\inf\{n\ge 2:X_i<\theta\}.$$ In either case $N$ has probability mass function given by 
\begin{equation}\label{eq:pn2}
\p(N=n)=\theta(1-\theta)^{n-2}, 0< \theta< 1, n=2,3,\ldots;
\end{equation}
note that this is simply a geometric random variable conditional on the success having occurred at trial 2 or later.  Clearly $N$ is dependent on the $X$ sequence.
\begin{prop}Under the CSUG model, the p.d.f.~of $Y$, defined by (1), is given by 
$$
g(y)=\frac{\theta}{(1-\theta)(1-y)^2},~0\le y \le 1-\theta.$$\end{prop}
\begin{figure}[tbp] 
  \centering
  \includegraphics[bb=132 223 480 569,width=4.68in,height=3.30in,keepaspectratio]{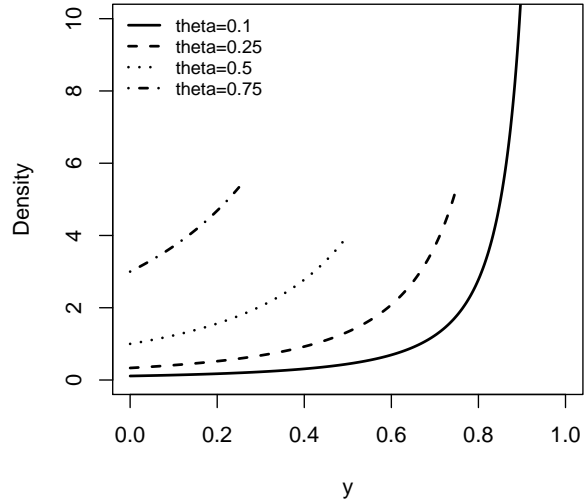}
  \caption{Plot of the CSUG maximum density for some values of $\theta$}
  \label{fig:csugy}
\end{figure}
\begin{proof}The conditional c.d.f.~of $Y$ given that $N=n$ is given by 
$$
\p(Y\le y|N=n)=\lr\frac{y}{1-\theta}\rr^{n-1},~n=2,3,\dotsc.$$
Taking the derivative, we see that the conditional density function is given by
$$
g(y|N=n)=\frac{n-1}{1-\theta}\lr\frac{y}{1-\theta}\rr^{n-2},~n=2,3,\ldots $$
Consequently, the p.d.f.~of $Y$ in the CSUG model is given by
\begin{eqnarray*}g(y)&=&\sum_{n=2}^{\infty}\theta(1-\theta)^{n-2}\times \frac{n-1}{1-\theta}\lr\frac{y}{1-\theta}\rr^{n-2}\\
&=&\frac{\theta}{1-\theta} \times \sum_{n=2}^{\infty}(n-1)y^{n-2}\\
&=&\frac{\theta}{(1-\theta)(1-y)^2}. \end{eqnarray*}  This completes the proof.\hfill
\end{proof}
\begin{prop}The p.d.f.~of $Z$ under the CSUG model is given by
$$
g(z)=\frac{\theta}{(1-\theta)z^2},~\theta \le z \le 1,~n=2,3,\ldots.$$\end{prop}
\begin{figure}[tbp] 
  \centering
  \includegraphics[bb=132 223 480 569,width=4.68in,height=3.30in,keepaspectratio]{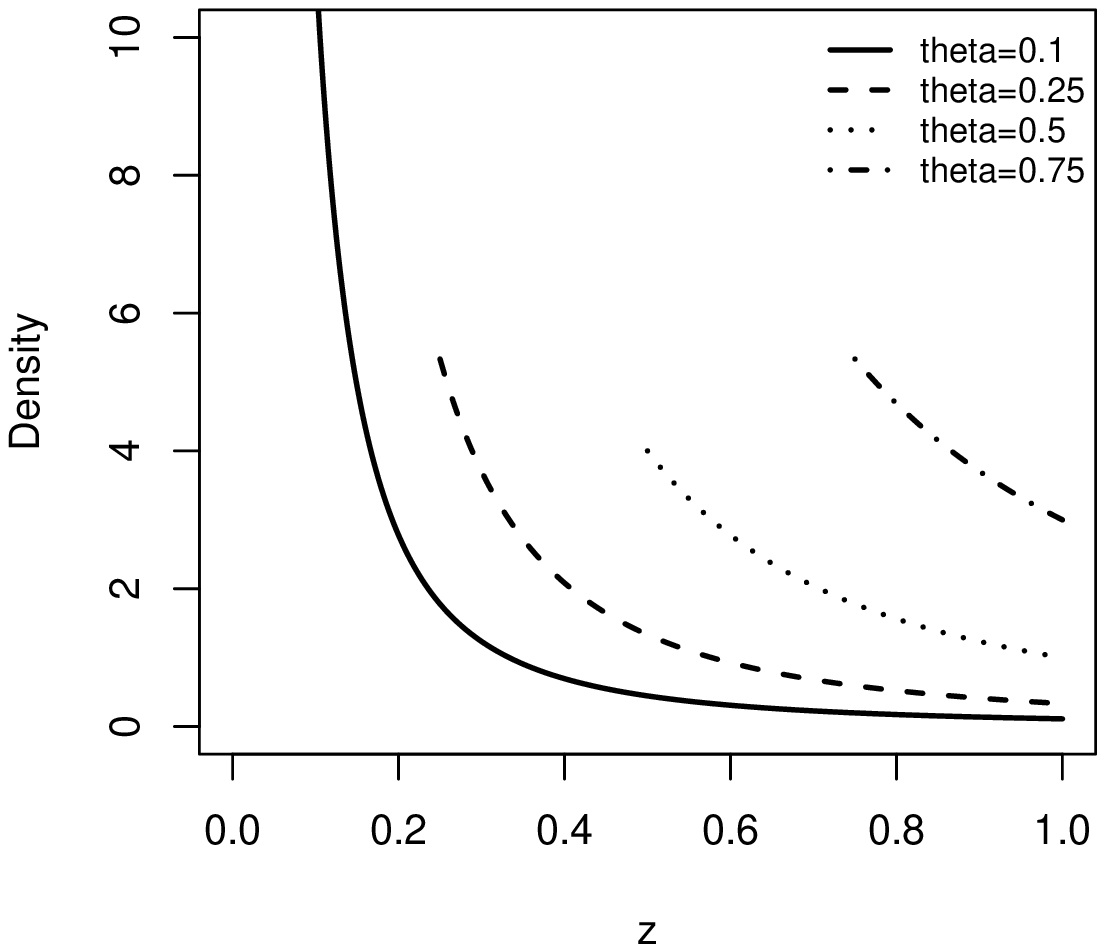}
  \caption{Plot of CSUG minimum density for some values of $\theta$}
  \label{fig:csugz}
\end{figure}

\begin{proof}The conditional cumulative distribution function of $Z$ given that $N=n$ is given by
$$
\p(Z\le z|N=n)=1-\p(Z>z|N=n)=1-\lr\frac{1-z}{1-\theta}\rr^{n-1},~n=2,3,\ldots$$
Thus, the conditional density function is given by 
$$
g(z|N=n)=\frac{n-1}{1-\theta}\lr\frac{1-z}{1-\theta}\rr^{n-2},~n=2,3,\dotsc, $$
which yields the p.d.f.~of $Z$ under the CSUG model as
\begin{eqnarray*}g(z)&=&\sum_{n=2}^{\infty}\theta(1-\theta)^{n-2}\times \frac{n-1}{1-\theta}\lr\frac{1-z}{1-\theta}\rr^{n-2}\\
&=&\frac{\theta}{1-\theta} \times \sum_{n=2}^{\infty}(n-1)(1-z)^{n-2}\\
&=&\frac{\theta}{(1-\theta)z^2}, \end{eqnarray*}
which finishes the proof.  \end{proof}

\begin{prop}If the random variable $Y$ has the ``CSUG maximum distribution" and $k\in {\mathbb N}$, then
$$\e(Y^k)=\frac{\theta}{1-\theta}\sum_{j=0}^{k}{k \choose j}\int_\theta^{1} (-u)^{j-2}du.$$ \end{prop}
\begin{proof}
\begin{eqnarray*}\e(Y^k)&=&\int_0^{1-\theta} y^k \times \frac{\theta}{(1-\theta)(1-y)^2}dy\\
&=&\frac{\theta}{1-\theta}\int_1^\theta \frac{(1-u)^k}{u^2}(-du)\\  
&=&\frac{\theta}{1-\theta}\int_\theta^1 \frac{\sum_{j=0}^{k}{k \choose j}(-u)^j}{u^2}du\\
&=&\frac{\theta}{1-\theta}\sum_{j=0}^{k}{k \choose j}\int_\theta^1 (-u)^{j-2}du, \end{eqnarray*}
as claimed.\hfill\end{proof} 
\begin{prop} The random variable $Y$ has mean and variance given, respectively, by
$$\e(Y)=\frac{\theta \ln\theta-\theta+1}{1-\theta}$$and $$\v(Y)=\frac{\theta^3-2\theta^2-\theta^2\ln^2\theta+\theta}{(1-\theta)^2}.$$ \end{prop}
\begin{proof}Using Proposition 3.3, we can directly compute the mean and variance by setting $k=1, 2$.  Notice that the variance of $Y$ is smaller than that of $Y$ under the SUG model, with an identical numerator term.  Also, the expected value is smaller under the CSUG model than in the SUG case.
\end{proof} 
\begin{prop} If the random variable $Z$ has the ``CSUG Minimum distribution" and $k\in {\mathbb N}$, then
$$\e(Z^k)=\frac{\theta}{1-\theta}\int_\theta^1z^{k-2}dz.$$ \end{prop} 
\begin{proof}
Routine, as before.\hfill\end{proof}
\begin{prop} The random variable $Z$ has mean and variance given, respectively, by
$$\e(Z)=\frac{-\theta\ln \theta}{1-\theta}$$and$$\v(Z)=\frac{\theta^3-2\theta^2-\theta^2\ln^2\theta+\theta}{(1-\theta)^2}.$$ \end{prop}
\begin{proof}A special case of Proposition 3.5; note that as in the SUG model, $\v(Y)=\v(Z)$. 
\hfill\end{proof}

\subsection{Parameter Estimation} The intermingling of polynomial and logarithmic terms makes method of moments estimation difficult in closed form, as in the SUG case.  However, if $\theta$ is unknown, the maximum likelihood estimate of $\theta$ can be found in a satisfying form, both in the CGUG maximum and CSUG minimum cases.  Suppose that $Y_1, Y_2, \ldots, Y_n$ form a random sample from the CSUG Maximum distribution with unknown $\theta$. Since the pdf of each observation has the following form:
\[ f(y|\theta) = \left\{ 
  \begin{array}{l l}
    \frac{\theta}{(1-\theta)(1-y)^2}, & \quad \text{for $0\le y\le 1-\theta$}\\
    0, & \quad \text{otherwise}
  \end{array} \right.\]
the likelihood function is given by 
\[ \ell(\theta) = \left\{ 
  \begin{array}{l l}
    (\frac{\theta}{1-\theta})^n\frac{1}{\prod_{i=1}^{n}(1-y_i)^2}, & \quad \text{for $0\le y_i \le 1-\theta~(i=1,2,\ldots,n)$}\\
    0, & \quad \text{otherwise}
  \end{array} \right.\]
The MLE of $\theta$ is a value of $\theta$, where $\theta \le 1-y_i~\mbox{for}~i=1, 2,\ldots, n$, which maximizes $\frac{\theta}{1-\theta}$. Let $\varphi(\theta)=\frac{\theta}{1-\theta}$. Since $\varphi'(\theta)\ge 0$, it follows that $\varphi(\theta)$ is a increasing function, which means the MLE is the largest possible value of $\theta$ such that $\theta \le 1-y_i~\mbox{for}~i=1,2,\ldots, n$. Thus, this value should be $1-\max(Y_1,\dots, Y_n)$, i.e., $\hat{\theta}=1-Y_{(n)}$.

Suppose next that $Z_1, Z_2, \dots, Z_n$ form a random sample from the CSUG minimum distribution. Since the pdf of each observation has the following form:
\[ f(z|\theta) = \left\{ 
  \begin{array}{l l}
    \frac{\theta}{(1-\theta)z^2}, & \quad \text{for $\theta \le z\le 1$}\\
    0 & \quad \text{otherwise,}
  \end{array} \right.\]
it follows that the likelihood function is given by 
\[ \ell(\theta) = \left\{ 
  \begin{array}{l l}
    (\frac{\theta}{1-\theta})^n\frac{1}{\prod_{i=1}^{n}z_i^2}, & \quad \text{for $\theta \le y_i \le 1~(i=1,2,\ldots,n)$}\\
    0 & \quad \text{otherwise.}
  \end{array} \right.\]
As above, it now follows that $\hat{\theta}=Y_{(1)}$.

\section{A Summary of Some Other Models}  The general scheme given by (3) and (4) is quite powerful.  As another example, suppose (using the example from Section 1) that 
\[p(n)=\frac{6}{\pi^2}\frac{1}{n^2}\] and $X\sim U[0,1]$.  Then it is easy to show that 
\[g(y)=\frac{6}{\pi^2}\frac{1}{y}\ln\lr\frac{1}{1-y}\rr, 0\le y\le 1,\]
and that $\e(Y)=\frac{6}{\pi^2}$.  (The expected value of $Y$ can also be calculated by using the identity
$\e(Y)=\e(\e(Y|N))$.  In this section, we collect some more results of this type, without proof:

\medskip

\noindent {UNIFORM-POISSON MODEL}  Here we let $X\sim U[0,1]$ and $p(n)=\frac{e^{-\lambda}\lambda^n}{(1-e^{-\lambda})n!}, n=1,2,\ldots$, so that $N$ follows a left-truncated Poisson distribution.

\begin{prop}
Under the Uniform-Poisson model,
\[g(y)=\frac{\lambda e^{-\lambda}e^{\lambda y}}{1-e^{-\lambda}};
g(z)=\frac{\lambda e^{-\lambda z}}{1-e^{-\lambda}};
\]
\[
\e(Y)=\frac{1}{1-e^{-\lambda}}-\displaystyle\frac{1}{\lambda};
\e(Z)=\displaystyle\frac{1}{\lambda}-\frac{e^{-\lambda}}{1-e^{-\lambda}};\]
\[\v(Y)=\displaystyle\frac{1}{\lambda^2}+\frac{1}{1-e^{-\lambda}}-\frac{1}{(1-e^{-\lambda})^2};
\v(Z)=\displaystyle\frac{1}{\lambda^2}-\frac{e^{-\lambda}}{\lambda(1-e^{-\lambda})}-\frac{e^{-2\lambda}}{(1-e^{-\lambda})^2};\]
$$M_Y(t)=\e(e^{ty})=\frac{\lambda e^{-\lambda}(e^{t+\lambda}-1)}{(t+\lambda)(1-e^{-\lambda})};M_Z (t)=\e(e^{tz})=\frac{\lambda(e^{t-\lambda}-1)}{(t-\lambda)(1-e^{-\lambda})}.$$
\end{prop}
In some sense, the primary motivation of this paper was to produce extreme value distributions that did not fall into the Beta family (such as $f(y)=nt^{n-1}$ for the maximum of $n$ i.i.d.~$U[0,1]$ variables).  A wide variety of non-Beta-based distributions may be found in \cite{bb}.  Can we add extreme value distributions to that collection?  In what follows, we use both the Beta families $B(2,2)$ and $B(1/2, 1/2)$, the arcsine distribution, and a ``Beyond Beta" distribution, the Topp-Leone distribution, as ``input variables" to make further progress in this direction.

\medskip
\noindent GEOMETRIC-BETA(2,2) MODEL.  Here $X\sim B(2,2)$ and $N\sim{\rm Geo}(\theta)$.  In this case we get 
\[g(y)=\frac{6y(1-y)\theta}{[1-(1-\theta)y^2(3-2y)]^2}\]
and
\[g(z)=\frac{6z(1-z)\theta}{[1-(1-\theta)(2z^3-3z^2+1)]^2}.\]

\medskip
\noindent POISSON-BETA(2,2) MODEL.  Here $X\sim B(2,2)$ and $N\sim{\rm Po0}(\theta)$, the Poisson($\theta$) distribution left-truncated at 0.  In this case we get 
\[g(y)=\frac{6\theta y(1-y)e^{-\theta(2y^3-3y^2+1)}}{1-e^{-\theta}}\]
and
\[g(z)=\frac{6\theta z(1-z)e^{-\theta(3z^2-2z^3)}}{1-e^{-\theta}}.\]

\medskip
\noindent GEOMETRIC-ARCSINE MODEL.  Here $X\sim B(1/2,1/2)$ and $N\sim{\rm Geo}(\theta)$.  In this case we get 
\[g(y)=\frac{\theta\pi^{-1}[y(1-y)]^{-1/2}}{[1-(1-\theta)\frac{2}{\pi}\arcsin\sqrt{y}]^2}\]
and
\[g(z)=\frac{\theta\pi^{-1}[z(1-z)]^{-1/2}}{[1-(1-\theta)(1-\frac{2}{\pi})\arcsin\sqrt{z}]^2}.\]

\medskip
\noindent POISSON-ARCSINE MODEL.  Here $X\sim B(1/2,1/2)$ and $N\sim{\rm Po0}(\theta)$.  Here we have  
\[g(y)=\frac{\theta \pi^{-1}[y(1-y)]^{-1/2}e^{-\theta(1-\frac{2}{\pi}\arcsin\sqrt{y})}}{1-e^{-\theta}}\]
and
\[g(z)=\frac{\theta \pi^{-1}[z(1-z)]^{-1/2}e^{-\frac{2\theta \arcsin\sqrt{z}}{\pi}}}{1-e^{-\theta}}.\]

\medskip
\noindent GEOMETRIC-TOPP-LEONE MODEL.  Here $X\sim TL(a)$ and $N\sim{\rm Geo}(\theta)$: 
\[g(y)=\frac{2a(1-y)y^{a-1}(2-y)^{a-1}\theta}{[1-(1-\theta)y^a(2-y)^a]^2}\]
and
\[g(z)=\frac{2a(1-z)z^{a-1}(2-z)^{a-1}\theta}{\{1-(1-\theta)[1-z^a(2-z)^a]\}^2}.\]

\medskip
\noindent POISSON-TOPP-LEONE MODEL.  $X\sim TL(a)$ and $N\sim{\rm Po0}(\theta)$:
\[g(y)=\frac{2\theta a (1-y)y^{a-1}(2-y)^{a-1}e^{-\theta[1-y^a(2-y)^a]}}{1-e^{-\theta}}\]
and
\[g(z)=\frac{2\theta a (1-z)z^{a-1}(2-z)^{a-1}e^{-\theta[z^a(2-z)^a]}}{1-e^{-\theta}}.\]

\section{Acknowledgments} The research of AG was supported by NSF Grants 1004624 and 1040928.

\end{document}